\documentclass[12pt, twoside]{article}
\usepackage{amsmath,amsthm,amssymb}
\usepackage[utf8]{inputenc}
\usepackage{times}
\usepackage{enumerate}
\usepackage{needspace}
\usepackage{tikz}
\usetikzlibrary{calc,positioning,shadows.blur,decorations.pathreplacing,patterns}
\usepackage{wrapfig}
\usepackage[export]{adjustbox}

 \newcommand*{\double}[2][.1ex]{%
  \mathrel{\vcenter{\offinterlineskip%
  \hbox{$#2$}\vskip#1\hbox{$#2$}}}}
\newcommand{\txtand}{{\mathop{\text{~and~}}}}

\newcommand{\Ccal}{{\mathcal{C}}}

\newcommand{\Ical}{{\mathcal{I}}}

\newcommand{\Ocal}{{\mathcal{O}}}

\newcommand{\Pbf}{{\mathbf{P}}}

\newcommand{\Wbf}{{\mathbf{W}}}

\newcommand{\restrict}{|}
\newcommand{\contract}{.}

\newcommand{\BS}{\backslash}

\newcommand{\sgn}{\mathrm{sgn}}

\newcommand{\rk}{\mathrm{rk}}

\newcommand{\Z}{\mathbb{Z}}

\newcommand{\N}{\mathbb{N}}

\newcommand{\disunion}{\mathbin{\dot{\cup}}}

\renewcommand{\phi}{\varphi}

\newcommand{\R}{\mathbb{R}}

\newcommand{\maparrow}{\longrightarrow}

\newcommand{\BSET}[1]{{\BS\left\{ #1 \right\}}}
\newcommand{\SET}[1]{{\left\{ #1 \right\}}}

\newcommand{\dSET}[1]{{\left\{ #1 \right\}_{\neq}}}

\newcommand{\COMMENT}[1]{}
\newcommand{\XXXCUTXXX}[1]{}

\newcommand{\ROMANENUM}{\renewcommand{\theenumi}{(\roman{enumi})}\renewcommand{\labelenumi}{\theenumi}}

\newcommand{\routesto}{\double{\rightarrow}}

\newcommand{\deftext}[2][]{\emph{#2}}
\newcommand{\deftextX}[2][]{\emph{#2}}

\newcommand{\PRFR}[1]{\ignorespaces}
\newcommand{\bm}{\ignorespaces}

\def\titlerunning#1{\gdef\titrun{#1}}
\makeatletter
\def\author#1{\gdef\autrun{\def\and{\unskip, }#1}\gdef\@author{#1}}
\def\address#1{{\def\and{\\\hspace*{18pt}}\renewcommand{\thefootnote}{}%
\footnote {#1}}%
\markboth{\autrun}{\titrun}}
\makeatother
\def\email#1{e-mail: #1}
\def\subjclass#1{{\renewcommand{\thefootnote}{}%
\footnote{\emph{Mathematics Subject Classification (2010):} #1}}}
\def\keywords#1{\par\medskip
\noindent\textbf{Keywords.} #1}

\newtheorem{theorem}{Theorem}[section]
\newtheorem{corollary}[theorem]{Corollary}

\newtheorem{lemma}[theorem]{Lemma}

\theoremstyle{definition}
\newtheorem{definition}[theorem]{Definition}

\newtheorem{example}[theorem]{Example}

\makeatletter
\renewenvironment{cases}[1][l]{\matrix@check\cases\env@cases{#1}}{\endarray\right.}
\def\env@cases#1{%
  \let\@ifnextchar\new@ifnextchar
  \left\lbrace\def\arraystretch{1.2}%
  \array{@{}#1@{\quad}l@{}}}

\makeatother

\begin{document}

% TITLE PAGE

\titlerunning{Heavy Arc Orientations of Gammoids}

\title{Heavy Arc Orientations of Gammoids}

\author{Immanuel Albrecht}

\date{\today}

\maketitle

\address{I.~Albrecht: %Hochschulstr.~36, D-01069~Dresden, Germany; \email{Immanuel.Albrecht@fernuni-hagen.de}
FernUniversität in Hagen, Fakultät für Mathematik und Informatik, Lehrgebiet für Diskrete Mathematik und Optimierung, D-58084~Hagen, Germany;
\email{Immanuel.Albrecht@fernuni-hagen.de}
}


% ABSTRACT

\begin{abstract}
 In this work, we introduce a \emph{purely combinatorial} way to obtain realizable orientations of a gammoid from a
 total order on the arc set of the digraph representing it, without first obtaining a matrix representing the gammoid over the reals.

%% Keywords are optional
\end{abstract}

This work is structured into two parts. First we develop a combinatorial method of obtaining an orientation of a \emph{cascade matroid} ---
i.e. of a gammoid that may be represented using an acyclic digraph. Then we introduce the method of \emph{lifting cycles} in order to
deal with gammoid representations that involve cycles.

% CONTENT

\section{Preliminaries}

In this work, we consider \emph{matroids} to be pairs $M=(E,\Ical)$ where $E$ is a finite set 
and $\Ical$ is a system of
independent subsets of $E$ subject to the usual axioms (\cite{Ox11}, Sec.~1.1).
%The family of bases of $M$ shall be denoted by $\Bcal(M)$,
%the family of flats of $M$ shall be denoted by $\Fcal(M)$.
The family of circuits of $M$ shall be denoted by $\Ccal(M)$.
If $M=(E,\Ical)$ is a matroid and $X\subseteq E$, then the restriction of $M$ to $X$
shall be denoted by $M\restrict X$ (\cite{Ox11}, Sec.~1.3),
and the contraction of $M$ to $X$ shall be denoted by $M\contract X$
(\cite{Ox11}, Sec.~3.1).
The \emph{dual matroid} of $M$ shall be denoted by $M^\ast$.

A \emph{signed} subset of $E$ shall be a map $X\colon E\mapsto \SET{-1,0,+1}$,
furthermore the \emph{positive} elements of $X$ shall be
\( X_+ = \SET{x\in E\mid X(x) = 1}, \)\label{n:xplus}
  the \deftext[negative elements of a signed subset]{negative} elements of $\bm X$ shall be
  \( X_- = \SET{x\in E\mid X(x)=-1},\)\label{n:xminus}
  the \deftext[support of a signed subset]{support} of $\bm X$ shall be
  \( X_\pm = \SET{x\in E\mid X(x)\not= 0},\)\label{n:xpm}
  and the \deftext[zero-set of a signed subset]{zero-set} of $\bm X$ shall be 
   $X_0 = E\BS X_\pm$.\label{n:xzero} The \deftext[negation of a signed subset]{negation} of $\bm X$
  shall be the signed subset \(-X\)\label{n:minusx} where $-X\colon E\maparrow \SET{-1,0,1},$ $e\mapsto -X(e)$.
 \emph{Oriented matroids} are
considered triples $\Ocal = (E,\Ccal,\Ccal^\ast)$ where $E$ is a finite set, $\Ccal$ is a family of signed circuits and $\Ccal^\ast$
is a family of signed cocircuits subject to the axioms of oriented matroids (\cite{BlVSWZ99}, Ch.~3). Every oriented matroid $\Ocal$
has a uniquely determined underlying matroid defined on the ground set $E$, which we shall denote by $M(\Ocal)$.
A matroid $M$ shall be \emph{orientable}, if there is an oriented matroid $\Ocal$ such that $M=M(\Ocal)$.

The notion of a \emph{digraph} shall be synonymous with what is described more 
precisely as \emph{finite simple directed graph} that may have some loops, i.e. a digraph is 
a pair $D=(V,A)$ where $V$ is a finite
set and $A\subseteq V\times V$ -- thus $\left| A \right| < \infty$. 
%Every digraph $D=(V,A)$ has a unique \emph{opposite digraph} $D^\opp = (V,A^\opp)$ where
%$(u,v)\in A^\opp$ if and only if $(v,u)\in A$.
All standard notions related to digraphs in this work are in
accordance with the definitions found in \cite{BJG09}. A \emph{walk} in $D=(V,A)$ is a non-empty
sequence $w = w_1 w_2 \ldots w_n$ of vertices $w_i\in V$ such that 
for each $1 \leq i < n$, $(w_i,w_{i+1})\in A$. By convention, we shall denote $w_n$ by $w_{-1}$.
Furthermore, the set of vertices traversed by a walk $w$ shall be denoted by $\left| w \right| = \SET{w_1,w_2,\ldots,w_n}$
and the set of all walks in $D$ shall be denoted by $\Wbf(D)$. Furthermore, the set of arcs traversed by $w$ shall be
denoted by $\left| w \right|_A = \SET{(w_1,w_2),(w_2,w_3),\ldots,(w_{n-1},w_n)}$. If $u,v\in \Wbf(D)$ with $u_{-1} = v_{1}$, then
$u.v = u_1 u_2 \ldots u_n v_2 v_3 \ldots v_m$, i.e. $u.v$ is the walk that traverses the arcs of $u$ and then the arcs of $v$.
A \emph{path} in $D=(V,A)$ is a walk $p = p_1 p_2 \ldots p_n$ such that $p_i = p_j$ implies $i=j$. 
The set of all paths in $D$ shall be denoted by $\Pbf(D)$. For $S,T\subseteq V$, an \emph{$S$-$T$-separator} in $D$
is a set $X\subseteq V$ such that every path $p\in\Pbf(D)$ from $s\in S$ to $t\in T$ has $\left| p \right| \cap V\not=\emptyset$.
A \emph{cycle} is a walk $c_1 c_2 \ldots c_n$ such that $n > 1$,  $c_1 = c_n$, and 
 $c_1 c_2 \ldots c_{n-1}$ is a path. An \emph{$S$-$T$-connector} shall be a routing $R\colon S' \routesto T$ with $S'\subseteq S$.

\begin{definition}\PRFR{Jan 22nd}
  Let $D = (V,A)$ be a digraph, and $X,Y\subseteq V$. A \deftext{routing} from $X$ to $Y$ in $D$ is a family of paths $R\subseteq \Pbf(D)$ such that
  \begin{enumerate}\ROMANENUM
    \item for each $x\in X$ there is some $p\in R$ with $p_{1}=x$,
    \item for all $p\in R$ the end vertex $p_{-1}\in Y$, and
    \item for all $p,q\in R$, either $p=q$ or $\left|p\right|\cap \left|q\right| = \emptyset$.%, and
        %\item all $p\in R$ are simple.
  \end{enumerate}
  We shall write $R\colon X\routesto Y$ in $D$ as a shorthand for ``$R$ is a routing from $X$ to $Y$
    in $D$'', and if no confusion is possible, \label{n:routing}
    we just write $X\routesto Y$ instead of $R$ and $R\colon X\routesto Y$.
    A routing $R$ is called \deftext{linking} from $X$ to $Y$, if it is a routing onto $Y$, i.e. whenever $Y = \SET{p_{-1}\mid p\in R}$.
\end{definition}

\begin{definition}\label{def:gammoid}\PRFR{Jan 22nd}
    Let $D = (V,A)$ be a digraph, $E\subseteq V$,
    and $T\subseteq V$. 
    The \deftext[gammoid represented by DTE@gammoid represented by $(D,T,E)$]{gammoid represented by $\bm{(D,T,E)}$} is defined to be the matroid $\Gamma(D,T,E)=(E,\Ical)$\label{n:GTDE}
     where
    \[ \Ical = \SET{X\subseteq E \mid \text{there is a routing } X\routesto T \text{ in D}}. \]
    The elements of $T$ are usually called \deftextX{sinks} in this context, although they are not required to be actual sinks of the digraph $D$. To avoid confusion, 
    we shall call the elements of $T$ \deftext{targets} in this work. A matroid $M'=(E',\Ical')$ is called \deftextX{gammoid}, if there is a digraph $D'=(V',A')$ and a set $T'\subseteq V'$ such that $M' = \Gamma(D',T',E')$.
    A gammoid $M$ is called \deftext{strict}, if there is a representation $(D,T,E)$ of $M$ with $D=(V,A)$ where $V=E$.
\end{definition}

Whenever $t\in T\cap E$, we have $\Gamma(D,T,E)\contract E\BSET{t} = \Gamma(D,T\BSET{t},E\BSET{t})$.

\begin{definition}\label{def:rspivot}\PRFR{Jan 22nd}
  Let $D=(V,A)$ be a digraph, $s\in V$ be a vertex of $D$, and $r \in V$ be a vertex such that
  $(r,s)\in A$ is an arc of $D$. The \deftext[pivot of a digraph]{$\bm r$-$\bm s$-pivot of $\bm D$} 
  shall be the digraph $D_{r\leftarrow s} = (V,A_{r\leftarrow s})$\label{n:digraphpivot} where
  $$ A_{r\leftarrow s} = \SET{(u,v)\in A ~\middle|~ u\not= r} 
   \cup \SET{(s,x) ~\middle|~ (r,x)\in A,\,x\not=s}.$$
\end{definition}
For example, pivoting $(r,s)$ in \includegraphics[scale=.75,valign=c]{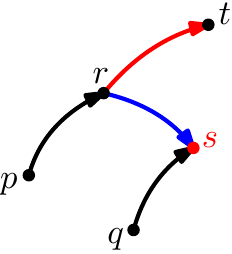} yields  \includegraphics[scale=.75,valign=c]{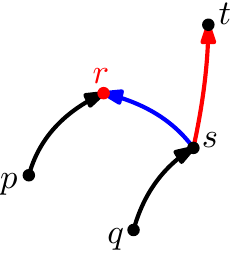}.

\begin{theorem}[\cite{M72}, The Fundamental Theorem (4.1.1)]\label{thm:fundamental}
  Let $D=(V,A)$ be a digraph, $T,E\subseteq V$, $s\in T$ which is sink in $D$, and $r\in V\BS T$ with $(r,s)\in A$.
  Then $\Gamma(D,T,E) = \Gamma(D_{r\leftarrow s}, T\BSET{s}\cup\SET{r}, E)$.
\end{theorem}
For a proof, see \cite{M72}.

\begin{lemma}\PRFR{Mar 7th}\label{lem:gammoidOrientable}
	Every gammoid $M=(E,\Ical)$ is orientable.
\end{lemma}
\begin{proof}\PRFR{Mar 7th}
	The class of gammoids is characterized as the closure of the class of transversal matroids under duality and minors (\cite{M72} Addendum from 21 Mar~1972; due to results from \cite{IP73}).
	Since every transversal matroid is representable over the reals, there is a set $T$ with $\left| T \right| = \rk_M(E)$ and there is a matrix $\mu\in \R^{E\times T}$,
	such that $M=M(\mu)$ (\cite{Ingleton1971}, Sec.~1). Thus $M$ is orientable since every matroid representable by elements of a real vector space
	has a natural orientation corresponding to the sign-patterns of the minimal non-trivial linear combinations of the zero vector (\cite{BlVSWZ99}, Sec.~1.2~(a)).
\end{proof}

\begin{definition}\PRFR{Feb 15th}
	Let $D=(V,A)$ be a digraph and $w\colon A\maparrow \R$.
	Then $w$ shall be called
	\deftext[indeterminate weighting of D@indeterminate weighting of $D$]{indeterminate weighting} of $\bm D$
  if $w$ is injective, and if there is no $r\in w[A]$ which may be expressed as
  \( r = a_0 + a_1 x_{1,1} x_{1,2}\cdots x_{1,n_1} + a_2 x_{2,1} x_{2,2}\cdots x_{2,n_2} + \ldots + a_m x_{m,1} x_{m,2}\cdots x_{m,n_m}\)
  with $a_i\in \Z$ and $x_{i,j} \in w[A] \BSET{r}$.
\end{definition}

Since the cardinality of the subset of the reals, which may be expressed %by integral linear combinations of products of elements of a finite subset $X\subseteq \R$,
as above, is countably infinite, a cardinality argument yields that every digraph $D=(V,A)$ has
an indeterminate weighting. Furthermore, if $w$ is an indeterminate weighting of $D$ and
$z\in \left( \Z\BSET{0} \right)^A$, then $w'$ with $w'(a) = z(a)\cdot w(a)$ is an indeterminate weighting of $D$, too.

\begin{definition}
  Let $D=(V,A)$ be a digraph, $w\colon A\maparrow \R$ be a map, and let $q=(q_i)_{i=1}^n\in \Wbf(D)$.
  We shall write \( \prod q \) in order to denote \( \prod_{i=1}^{n-1} w\left( \vphantom{A^A}(q_i,q_{i+1}) \right).\)
\end{definition}

%\needspace{8\baselineskip}
 \begin{lemma}[Lindström \cite{Li73}]\label{lem:lindstrom}\PRFR{Feb 15th}
 	Let $D=(V,A)$ be an acyclic digraph, $n\in \N$, $S=\dSET{s_1,s_2,\ldots,s_n}\subseteq V$ and $T=\dSET{t_1,t_2,\ldots,t_n}\subseteq V$ be equicardinal, % subsets of $V$, and let 
  and $w\colon A\maparrow \R$ be an indeterminate weighting of $D$. Furthermore,
 	$\mu\in \R^{V\times V}$ shall be the matrix defined by the  equation
 	\[ \mu(u,v) = \sum_{p\in \Pbf(D;u,v)} \prod p .\]
 	where $\Pbf(D;u,v) = \SET{p\in \Pbf(D)\mid  p_1 = u \txtand p_{-1}=v}$.
 	Then
 	\[ \det \left( \mu\restrict S\times T \right) = \sum_{L\colon S\routesto T} \left( \sgn(L)  
 	\prod_{p\in L} \left( \prod p \right) \right) \]
 	where the sum ranges over all linkings $L$ that route $S$ to $T$ in $D$; and
  where $\sgn(L) = \sgn(\sigma)$ for the unique permutation $\sigma\in \mathfrak{S}_n$ with
 	the property that for every $i\in\SET{1,2,\ldots,n}$ there is a path $p\in L$ with $p_1 = s_i$ and 
 	$p_{-1} = t_{\sigma(i)}$.
 	Furthermore, \[  \det \left( \mu\restrict S\times T \right) = 0 \] if and only if
 	there is no linking from $S$ to $T$ in $D$.
 \end{lemma}

I.M.~Gessel and X.G.~Viennot gave a nice bijective proof in \cite{GV89p}.

\section{Heavy Arcs and Routings}

\begin{lemma}\label{lem:CRAMERsrule}\PRFR{Mar 7th}
	Let $E$ and $T$ be finite sets, and let $\mu\in \R^{E\times T}$ be a matrix, 
	and $M = M(\mu)$ be the matroid represented by $\mu$ over $\R$. 
	Further, let $\Ocal = (E,\Ccal,\Ccal^\ast) = \Ocal(\mu)$
	be the oriented matroid obtained from $\mu$,
	let $C\in \Ccal(M)$ % be a circuit of $M$ and let 
  and $c\in C$. % be an arbitrary element of that circuit.
	Let $T_0\subseteq T$ such that $\det(\mu \restrict (C\BSET{c})\times T_0) \not= 0$.
	Consider the signed subset $C_c$ of $E$ with
	 \[ C_c(e) = \begin{cases}[r]
	 	0 & \quad \text{if~} e\notin C,\\
	 	-1 & \quad \text{if~} e = c,\\
	 	\sgn\left( \frac{\det (\nu_e)}{\det(\mu \restrict (C\BSET{c})\times T_0)}\right) & \quad \text{otherwise}
	 \end{cases} \]
	 where \[ \nu_e \colon C\BSET{c}\times T_0 \maparrow \R,\quad (x,t)\mapsto \begin{cases}
	 		\mu(c,t) & \quad \text{if~} x = e,\\
	 		\mu(x,t) & \quad \text{otherwise.}
	 \end{cases} 
	 \]
	 Then $C_c \in \Ccal$.
\end{lemma}

\begin{proof}\PRFR{Mar 7th}
	By \textsc{Cramer}'s rule we obtain that
	\[ \mu_c = \sum_{e\in C\BSET{c}} \frac{\det (\nu_e)}{\det(\mu \restrict (C\BSET{c})\times T_0)} \cdot \mu_e \]
  where $\mu_i$ denotes the row of $\mu$ with index $i$, i.e. $\mu_i = \mu(i, \bullet)$.
	Therefore, 
	\[ -\mu_c + \sum_{e\in C\BSET{c}} \frac{\det (\nu_e)}{\det(\mu \restrict (C\BSET{c})\times T_0)} \cdot \mu_e = 0 \]
	is a non-trivial linear combination of the zero vector. Clearly $C_c$ consists of the signs of the corresponding coefficients and therefore $C_c\in \Ccal$ is an orientation of $C$ with respect to $\Ocal(\mu)$. 
\end{proof}

 \begin{definition}\label{def:heavyArcSignature}\PRFR{Mar 7th}
 	Let $D=(V,A)$ be a digraph, let $\sigma\colon A\maparrow \SET{-1,1}$ be a map and let $\ll$ be a binary relation on $A$.
 	We shall call $(\sigma, \ll)$ \label{n:sigmaLL} a \deftext[heavy arc signature of a digraph]{heavy arc signature of $\bm D$}, 
 	if $\ll$ is a linear order on $A$.
 \end{definition}

 %\begin{definition}
 %	Let $D=(V,A)$ be a digraph, $X,Y\subseteq V$, $\ll\subseteq A\times A$ be a linear order on the arcs of $D$, and
 %	let $L_1,L_2\colon X\routesto Y$ be two distinct linkings from $X$ to $Y$ in $D$.
 %	Then $L_1$ shall be called \deftext[lighter linking]{lighter} than $L_2$,
 %	if there is some $a_2 \in A_2\BS A_1$ such that for all
 %	$a_1 \in A_1\BS A_2$, $a_1 \ll a_2$ holds;
 %	where $A_1 = \bigcup_{p\in L_1} \left| p \right|_A$ and
 %	$A_2 = \bigcup_{p\in L_2} \left| p \right|_A$.
 %	In other words, $L_1$ is lighter than $L_2$ if there is a path in $L_2$, which traverses an arc $a_2$ that is 
 %	not
 %	traversed by a path in $L_1$, such that every arc $a_1$ traversed by a path in $L_1$
 %	but not by a path in $L_2$ is smaller with respect to $\ll$.
 %	In this case, we shall write $L_1 \llless L_2$.\label{n:llless}
 %\end{definition}
 \begin{definition}\label{def:routingOrder}\PRFR{Mar 7th}
 	Let $D=(V,A)$ be a digraph and $(\sigma, \ll)$ be a heavy arc signature of $D$.
 	The \deftext[induced routing order]{$\bm( \bm \sigma \bm, \bm \ll \bm)$-induced routing order of $\bm D$}
 	shall be the linear order $\lll$ on the family of routings of $D$, where $Q \lll R$ holds\label{n:llless}
 	if and only if the $\ll$-maximal element $x$ of the symmetric difference $Q_A \bigtriangleup R_A$ has
 	the property $x\in R_A$, where $Q_A = \bigcup_{p\in Q} \left| p \right|_A$ and $R_A = \bigcup_{p\in R} \left| p \right|_A$.
 \end{definition}

 	Clearly, $\llless$ is a linear order on all routings in $D$, because every routing $R$ in $D$ is uniquely determined by its set of
 	traversed arcs $R_A$.

\begin{definition}\PRFR{Mar 7th}
	Let $D=(V,A)$ be a digraph, and let $(\sigma,\ll)$ be a heavy arc signature of $D$.
	Let $R\colon X\routesto Y$ be a routing in $D$ where $X=\dSET{x_1,x_2,\ldots,x_n}$
	and $Y=\dSET{y_1,y_2,\ldots,y_m}$ are implicitly ordered.
	The \deftext[sign of a routing]{sign of $\bm R$ with respect to $\bm ( \bm \sigma\bm,\bm \ll \bm)$} shall be\label{n:sgnsigma}
	\[ \sgn_\sigma (R) = \sgn(\phi) \cdot \left( \prod_{p\in R,\,a\in \left| p \right|_A} \sigma (a)  \right) \]
	where $\phi\colon \SET{1,2,\ldots,n}\maparrow \SET{1,2,\ldots,m}$ is the unique map such that 
	for all $i \in \SET{1,2,\ldots,n}$ there is a path $p\in R$
	with $p_1 = x_i$ and $p_{-1} = y_{\phi(x)}$; and where
	\[ \sgn(\phi) = {(-1)}^{\left| \SET{\vphantom{A^A}(i,j)~\middle|~i,j\in \SET{1,2,\ldots,n}\colon\, i < j \txtand \phi(i) > \phi(j)}\right|} . \qedhere\]
\end{definition}

\begin{definition}\label{def:Csigmac}\PRFR{Mar 7th}
	Let $D=(V,A)$ be a digraph such that $V=\dSET{v_1,v_2,\ldots,v_n}$ is implicitly ordered,
	 $(\sigma,\ll)$ be a heavy arc signature of $D$, and let $T,E\subseteq V$
	be subsets that inherit the implicit order of $V$.
	Furthermore, let  $M=\Gamma(D,T,E)$ be the corresponding gammoid, and let $C\in \Ccal(M)$
	be a circuit of $M$ 
	 such that $C = \dSET{c_1,c_2,\ldots,c_m}$ inherits its implicit order from $V$;
	 and let $i\in\SET{1,2,\ldots,m}$. 
%	order of $E$ with respect to $\sigma$,
%	and $c_i\in C$. 
%
	The \deftext[heavy arc circuit signature]{signature of $\bm C$ with respect to $\bm M$, $\bm i$, and $\bm (\bm \sigma \bm, \bm \ll \bm)$}
	shall be the signed subset $C_{(\sigma,\ll)}^{(i)}$ of $E$ where
%	is defined to be $C_{\sigma c_i} \in \sigma E$ where for all $e\in E$\label{n:Csigmac}
	\[ C_{(\sigma,\ll)}^{(i)}(e) = \begin{cases}[r]
							0 &\quad \text{if~}e\notin C,\\
							- \sgn_{\sigma}(R_{i}) &\quad \text{if~}e = c_i,\\
							(-1)^{i-j+1} \cdot \sgn_{\sigma}(R_{j})   &\quad \text{if~}e = c_j\not= c_i,
						\end{cases} \]
	and where for all $k\in \SET{1,2,\ldots,m}$
	\[ R_k = \max_{\llless} \SET{R \mid R\colon C\BSET{c_k}\routesto T\text{~in~}D} \]
	denotes the unique $\llless$-maximal routing from $C\BSET{c_k}$ to $T$ in $D$.
\end{definition}

	Note that the factors $(-1)^{i-j+1}$ in Definition~\ref{def:Csigmac} do not appear explicitly in Lemma~\ref{lem:CRAMERsrule},
	where
	$\nu_e$ is obtained from the restriction $\mu\restrict (C\BSET{c})\times T_0$ 
	by replacing the values in row $e$ with the values of $\mu_c$.
	We have to account for the number of row transpositions that are needed to turn $\nu_e$ into the restriction $\mu\restrict(C\BSET{e})\times T_0$, which depends on the position of $e=c_j$ relative to $c=c_i$ with respect to the implicit order of $V$.

\begin{definition}\label{def:heavyArcWeighting}\PRFR{Mar 7th}
	Let $D=(V,A)$ be a digraph and $(\sigma, \ll)$ a heavy arc signature of $D$, and let $w\colon A\maparrow \R$ be an indeterminate weighting of $D$.
	We say that $w$ is a \deftext[heavy arc weighting]{$\bm( \bm\sigma\bm,\bm\ll\bm)$-weighting of $\bm D$} if, for all $a\in A$,
	the inequality $\left| w(a) \right| \geq 1$,
	the strict inequality
	\[\sum_{L\subseteq \SET{x\in A~\middle|~ x \ll a,\,x\not=a}} \left( \prod_{x\in L} \left| w(x)  \right| \right)  < \left| w(a) \right|,  \]
	and the equality 
	\( \sgn(w(a)) = \sigma(a) \)
	hold.
\end{definition}

%\needspace{4\baselineskip}
\begin{lemma}\label{lem:ExistenceOfHeavyArcWeighting}\PRFR{Mar 7th}
	Let $D=(V,A)$ be a digraph and $(\sigma, \ll)$ be a heavy arc signature of $D$.
	There is a $(\sigma,\ll)$-weighting of $D$.
\end{lemma}
\begin{proof}\PRFR{Mar 7th}
	Let $w\colon A\maparrow \R$ be an indeterminate weighting of $D$.
	For every $\zeta \in \Z^A$ and every $\tau \in \SET{-1,1}^A$, the map
	$w_{\zeta,\tau}\colon A\maparrow \R$, which is defined by the equation
   $$w_{\zeta,\tau}(a) = \tau(a)\cdot \frac{w(a)}{\sgn(w(a))} + \tau(a) \cdot \zeta(a)$$ is an indeterminate weighting of $D$, too.
	Now, let $\zeta\in \Z^A$, such that for all $a\in A$ we have the following recurrence relation
	\[ \zeta(a) =  \left\lceil  \sum_{L\subseteq \SET{x\in A ~\middle|~ x \ll a,\,x\not=a}} \left( \prod_{x\in L} \left(\vphantom{A^1} \left| w(x) \right| + \zeta(x) \right) \right) \right\rceil.\]
	The map $\zeta$ is well-defined by this recurrence relation because $\left| A \right| < \infty$ and therefore there is a $\ll$-minimal element $a_0$ in $A$. In particular, we have the equation
   $\zeta(a_0) = \prod_{x\in L=\emptyset} \left(\vphantom{A^1} \left| w(x) \right| + \zeta(x) \right) = 1$.
	Thus $w_{\zeta,\sigma}$ is a $(\sigma,\ll)$-weighting of $D$. Clearly,
	\begin{align*}
		  \sgn\left( w_{\zeta,\sigma}(a)  \right) & = \sgn\left( \sigma(a) \cdot \frac{w(a)}{\sgn(w(a))} + \sigma(a) \cdot \zeta(a) \right) 
		  \\ & 
		  = \sgn\left( \vphantom{a^1}\sigma(a) \right)\cdot\sgn\left( \frac{w(a)}{\sgn(w(a))} +  \zeta(a) \right)\\& 
		  = \sigma(a) \cdot 1 = \sigma(a)
		  \end{align*}
	holds for all $a\in A$. Furthermore, we have
	\begin{align*}
		\left| w_{\zeta,\sigma}(a) \right| & = \left| \sigma(a) \cdot \frac{w(a)}{\sgn(w(a))} + \sigma(a) \cdot \zeta(a) \right| \\
		& > \left| \zeta(a) \right| %\\
		 = \left\lceil  \sum_{L\subseteq \SET{x\in A ~\middle|~ x \ll a,\,x\not=a}} \left( \prod_{x\in L} \left(\vphantom{A^1} \left| w(x) \right| + \zeta(x) \right) \right) \right\rceil \\
		& \geq  \sum_{L\subseteq \SET{x\in A~\middle|~ x \ll a,\,x\not=a}} \left( \prod_{x\in L} \left|  w_{\zeta,\sigma}(x) \right| \right). \qedhere
	\end{align*}
\end{proof}

%\needspace{5\baselineskip}

\begin{lemma}\label{lem:lllMaximalRoutingsHaveCommonEnd}\PRFR{Mar 7th}
	Let $D=(V,A)$ be a digraph, $(\sigma,\ll)$ be a heavy arc weighting of $D$, $E,T\subseteq V$,
	$C\in \Ccal(\Gamma(D,T,E))$ be a circuit in the corresponding gammoid, and let $c,d\in C$.
	Furthermore, let $R_c \colon C\BSET{c} \routesto T$ and $R_d \colon C\BSET{d} \routesto T$ be 
	the $\lll$-maximal routings in $D$.
	Then \( \SET{p_{-1} ~\middle|~ p\in R_c} = \SET{p_{-1} ~\middle|~ p\in R_d} \)
	holds.
\end{lemma}
\begin{proof}\PRFR{Mar 7th}
	Let $S$ be a $C$-$T$-separator of minimal cardinality in $D$, i.e. a $C$-$T$-separator with
	$\left| S \right| = \left| C \right| -1$ (\textsc{Menger}'s Theorem). Since $R_c$ and $R_d$ are both $C$-$T$-connectors with maximal cardinality,
	we obtain that for every $s\in S$ there is a path $p_c^s \in R_c$ and a path $p_d^s \in R_d$ such that $s\in \left| p_c^s \right|$ and
	$s\in \left| p_d^s \right|$, thus there are paths $l_c^s,l_d^s,r_c^s,r_d^s\in \Pbf(D)$ such that
	$p_c^s = l_c^s . r_c^s$ and
	 $p_d^s = l_d^s . r_d^s$ with $\left( {r_c^s} \right)_1 = \left( {r_d^s} \right)_1 = s$. Now let $R_c^S = \SET{r_c^s ~\middle|~ s\in S}$
	and $R_d^S = \SET{r_d^s ~\middle|~ s\in S}$, clearly both $R_c^S$ and $R_d^S$ are routings from $S$ to $T$ in $D$.
	Assume that $R_c^S \not= R_d^S$, then we have $R_c^S \lll R_d^S$ --- without loss of generality, by possibly
	switching names for $c$ and $d$. Then $Q = \SET{l_c^s . r_d^s ~\middle|~ s\in S}$ is a routing from $C\BSET{c}$ to $T$ in $D$.
	But for the symmetric differences we have the equality 
	\[ \left( \bigcup_{p\in Q} \left| p \right|_A \right) \bigtriangleup \left( \bigcup_{p\in R_c} \left| p \right|_A \right) 
	= \left( \bigcup_{p\in R_d^S} \left| p \right|_A \right) \bigtriangleup \left( \bigcup_{p\in R_c^S} \left| p \right|_A \right),
	\]
	which implies $R_c \lll Q$, a contradiction to the assumption that $R_c$ is the $\lll$-maximal routing from $C\BSET{c}$ to $T$.
	Thus $R_c^S = R_d^S$ and the claim of the lemma follows.
\end{proof}

%\needspace{4\baselineskip}
\begin{lemma}\label{lem:acyclicOrientation}\PRFR{Mar 7th}
	Let $D=(V,A)$ be an acyclic digraph where $V$ is implicitly ordered, $(\sigma,\ll)$ be a heavy arc signature of $D$, and $T,E\subseteq V$.
	Then there is a unique oriented matroid $\Ocal=(E,\Ccal,\Ccal^\ast)$ where
	 \[ \Ccal = \SET{ \pm C_{(\sigma,\ll)}^{(1)} ~\middle|~ C\in \Ccal(\Gamma(D,T,E))}.\]
\end{lemma}
\begin{proof}\PRFR{Mar 7th}
	Let $M=\Gamma(D,T,E)$, and let $w\colon A\maparrow \R$ be a $(\sigma,\ll)$-weighting of $D$. % which exists due to Lemma~\ref{lem:ExistenceOfHeavyArcWeighting}.
	Furthermore, let $\mu\in \R^{E\times T}$ be the matrix defined as in the Lindström Lemma~\ref{lem:lindstrom}, with respect to the
	$(\sigma,\ll)$-weighting $w$ and the implicit order on $V$.
  % Theorem~\ref{thm:gammoidOverR} along with its proof yields that we have $M = M(\mu)$.
  The second statement of the Lindström Lemma yields $M = M(\mu)$.
	Let $\Ocal = \Ocal(\mu) = (E,\Ccal_\mu,\Ccal_\mu^\ast)$ be the oriented matroid that arises from $\mu$, thus $M(\Ocal) = M(\mu)$.
  % holds (Corollary~\ref{cor:MOmuEQMmu}). 
  %We show that $\Ccal_\mu = \Ccal$. 
  It suffices to prove
	 that for all $C\in \Ccal(M)$, all $D \in \Ccal_\mu$ with $D_\pm = C$, and all $D'\in \Ccal$
	with $D_\pm = C$ we have $D \in \SET{D',-D'}$. 
	Now, let $C\in \Ccal(M)$ and let $C=\dSET{c_1,c_2,\ldots,c_k}$ implicitly ordered respecting the implicit order of $V$.
	The claim follows if  $D(c_1)D(c_j) = D'(c_1)D'(c_j)$ holds for all $j\in \SET{2,3,\ldots,k}$.
	Let $T_0 \subseteq T$ be the target vertices onto which the $\lll$-maximal 
	and $\left| \cdot \right|$-maximal $C$-$T$-connectors link in $D$ 
	(Lemma~\ref{lem:lllMaximalRoutingsHaveCommonEnd}).
	From Lemma~\ref{lem:CRAMERsrule} we obtain that
	\begin{align*}
		D(c_1)D(c_j) & = -1\cdot \sgn\left( \frac{\det (\nu_j)}{\det(\mu \restrict (C\BSET{c_1})\times T_0)}\right)\\
		& = -\sgn(\det(\nu_j))\cdot \sgn(\mu \restrict (C\BSET{c_1})\times T_0)) 
	\end{align*}
	where \[ \nu_j \colon C\BSET{c_1}\times T_0 \maparrow \R,\quad (x,t)\mapsto \begin{cases}[r]
	 		\mu(c_1,t) & \quad \text{if~} x = c_j,\\
	 		\mu(x,t) & \quad \text{otherwise.}
	 \end{cases} 
	 \]
	 Observe that $\nu_j$ arises from the restriction $\mu\restrict C\BSET{c_j}\times T_0$ by a 
	 row-permutation, which has at most one non-trivial cycle,
	 and this cycle then has length $j-1$, \linebreak therefore $\det(\nu_j) = (-1)^{j-2} \det \left( \mu\restrict C\BSET{c_j}\times T_0 \right)$ holds,
	 so we get
	 $
		D(c_1)D(c_j)  =  (-1)^{1-j}\sgn\left( \det \left( \mu\restrict C\BSET{c_j}\times T_0 \right) \right)\cdot \sgn(\mu \restrict (C\BSET{c_1})\times T_0)) $.
	 We further have
	 $
	 	D'(c_1)D'(c_j)  = (-1)^{j+1}\cdot \sgn_{\sigma}(R_1)\cdot \sgn_{\sigma}(R_j)
	 $
	 where for all $i\in \SET{1,2,\ldots,k}$ the symbol
	\( R_i \) %= \max_{\llless} \SET{R \mid R\colon C\BSET{c_i}\routesto T\text{~in~}D} \)
	denotes the unique $\llless$-maximal routing from $C\BSET{c_i}$ to $T$ in $D$.
	By the Lindström Lemma~\ref{lem:lindstrom} we obtain that for all $i\in\SET{1,2,\ldots,k}$
	the equation
	\begin{align*}
		\det\left( \mu\restrict C\BSET{c_i}\times T_0 \right) & =  \sum_{R\colon C\BSET{c_i}\routesto T_0} \left( \sgn(R)  
 		\prod_{p\in R} \left( \prod_{a\in \left| p \right|_A} w(a) \right) \right) 
 	\end{align*}
 	holds,
	where $\sgn(R)$ is the sign of the permutation implicitly given by the start and end vertices of the paths in $R$, 
	both with respect to the implicit order on $V$.
	Since $w$ is a $(\sigma,\ll)$-weighting, we have
	\[ \left| \sum_{R\colon C\BSET{c_i}\routesto T_0,\,R\not= R_i} \left( \sgn(R)  
 		\prod_{p\in R} \left( \prod_{a\in \left| p \right|_A} w(a) \right) \right) \right| < \left| w(a_i)  \right| \]
 	where $a_i \in \bigcup_{p\in R_i} \left| p \right|_A$ is the $\ll$-maximal arc in the $\lll$-maximal routing $R_i$ from
 	$C\BSET{c_i}$ to $T_0$ in $D$. Therefore the sign of $\det\left( \mu\restrict C\BSET{c_i}\times T_0 \right)$ is determined
 	by the sign of the summand that contains $w(a_i)$ as a factor, which is the summand that corresponds to $R = R_i$.
	Therefore
	\begin{align*}
		\sgn \left( \det\left( \mu\restrict C\BSET{c_i}\times T_0 \right) \right) %& = 
		 %\sgn \left( \sgn(R_i)\prod_{p\in R_i,\,a\in \left| p \right|_A} w(a) \right) \\
		 & = \sgn(R_i) \prod_{p\in R_i,\,a\in \left| p \right|_A} \sgn(w(a)) \\
		  = \sgn(R_i) \prod_{p\in R_i,\,a\in \left| p \right|_A} \sigma(a) 
		 & = \sgn_\sigma (R_i).
	\end{align*}
	So we obtain
	\begin{align*}
		D(c_1)D(c_j) & =  (-1)^{1-j} \sgn_\sigma(R_1) \cdot \sgn_{\sigma}(R_j) %\\
		           %  & =   (-1)^{j+1}\cdot \sgn_{\sigma}(R_1)\cdot \sgn_{\sigma}(R_j)
                = D'(c_1)D'(c_j). \qedhere
	\end{align*}
\end{proof}

%\noindent Unfortunately, we cannot omit the assumption that $D$ is an acyclic digraph.

%\needspace{6\baselineskip}

\vspace*{-\baselineskip} %Remove the line space created by the tilde below
\begin{wrapfigure}{r}{6cm}
\vspace{\baselineskip}
\begin{centering}~~%move the picture slightly to the right
\includegraphics{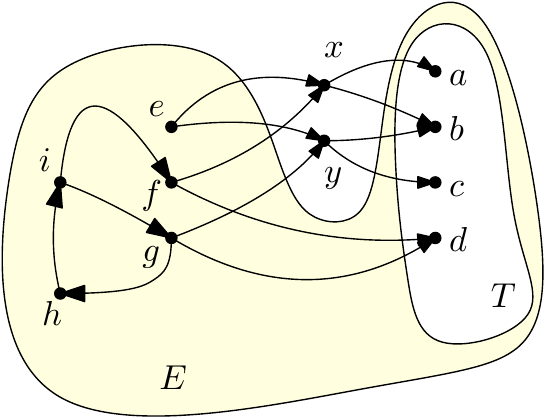}
\end{centering}%
\vspace*{-\baselineskip} %make the picture more tightly cropped
\end{wrapfigure}
~ %The tilde creates a new dummy paragraph. WHY IS THAT NEEDED? -> would increase the space %
  % before the ex. environment. THE NEXT FREE LINE IS ESSENTIAL!
 
\begin{example}
	We consider the digraph \linebreak $D=(V,A)$ with the implicitly ordered vertex set $V=\dSET{a,b,c,d,e,f,g,h,i,x,y}$, 
	and $A$ as depicted on the right. Let $T=\SET{a,b,c,d}$.
	Clearly, $\Wbf(D)$ contains the cycle walk $ghig$. Let $(\sigma,\ll)$ be the heavy arc signature of $D$ where
	$\sigma(a) = 1$ for all $a\in A$, and where $a_1 \ll a_2$ if the tuple $a_1$ is less than the tuple $a_2$ with respect to the
	lexicographic order on $V\times V$ derived from the implicit order of the vertex set.
	Let $C_1 = \SET{f,g,i}$, $C_2 = \SET{d,e,f,i}$, $C_{f} = \SET{d,e,g,i}$. Clearly $C_1,C_2,C_{f}\in \Ccal(\Gamma(D,T,E))$.
	The following routings are $\lll$-maximal among all routings in $D$ with the same set of initial vertices and with targets in $T$.
  \allowdisplaybreaks
	\begin{align*}
		R_{\SET{f,g}} & = \SET{fxb, gyc}  & \sgn_{\sigma}\left( R_{\SET{f,g}} \right) & = +1\\
		R_{\SET{f,i}} & = \SET{fxb, igyc} & \sgn_{\sigma}\left( R_{\SET{f,i}} \right) & = +1\\
		R_{\SET{g,i}} & = \SET{gyc, ifxb} & \sgn_{\sigma}\left( R_{\SET{g,i}} \right) & = -1\\
		R_{\SET{d,e,f}} & = \SET{d,eyc,fxb} & \sgn_{\sigma}\left( R_{\SET{d,e,f}} \right) & = -1\\
		R_{\SET{d,e,i}} & = \SET{d,exb,igyc} & \sgn_{\sigma}\left( R_{\SET{d,e,i}} \right) & = +1\\
		R_{\SET{d,f,i}} & = \SET{d,fxb,igyc} & \sgn_{\sigma}\left( R_{\SET{d,f,i}} \right) & = +1\\
		R_{\SET{e,f,i}} & = \SET{exb,fd,igyc} & \sgn_{\sigma}\left( R_{\SET{e,f,i}} \right) & = -1\\
		R_{\SET{d,e,g}} & = \SET{d,eyc,ghifxb} & \sgn_{\sigma}\left( R_{\SET{d,e,g}} \right) & = -1\\
		R_{\SET{d,g,i}} & = \SET{d,gyc,ifxb} & \sgn_{\sigma}\left( R_{\SET{d,g,i}} \right) & = -1\\
		R_{\SET{e,g,i}} & = \SET{exb,gyc,ifd} & \sgn_{\sigma}\left( R_{\SET{e,g,i}} \right) & = +1\\
	\end{align*}
	Now let us calculate the signatures of $C_1$, $C_2$, and $C_f$ according to Definition~\ref{def:Csigmac}.
	We obtain \(
		\left( C_{1} \right)^{(1)}_{(\sigma,\ll)} = \SET{f,g,-i}
		\), \(
		\left( C_{2} \right)^{(1)}_{(\sigma,\ll)} = \SET{d,e,-f,-i}
		,\) and  
		\( \left( C_{f} \right)^{(1)}_{(\sigma,\ll)} = \SET{-d,-e,-g,-i}.
		\)
	This clearly violates the oriented strong circuit elimination (\cite{BlVSWZ99}, Thm.~3.2.5): if we eliminate $f$ from $\left( C_{1} \right)^{(1)}_{(\sigma,\ll)}$ 
	and $\left( C_{2} \right)^{(1)}_{(\sigma,\ll)}$, then the resulting signed circuit must have opposite signs for $d$ and $i$,
	but $d$ and $i$ have the same sign with respect to $\left( C_{f} \right)^{(1)}_{(\sigma,\ll)}$. Therefore we see that
	the assumption, that $D$ is
	acyclic, cannot be dropped from Lemma~\ref{lem:acyclicOrientation}.
\end{example}

\section{Dealing with Cycles in Digraphs}

\noindent We can still use the construction involved in Lemma~\ref{lem:acyclicOrientation} to obtain an orientation 
for a representation $(D,T,E)$ of a gammoid $M$ where $D$ is not acyclic,
but we first have to construct something we call complete lifting of $D$,
which yields an acyclic representation of a co-extension $M'$ of $M$. Finally, we may obtain an orientation of $M$
by contraction of a heavy arc orientation of $M'$.

\begin{definition}\PRFR{Feb 15th}
  Let $D=(V,A)$ be a digraph, $x,t\notin V$ be distinct new elements, and let $c=(c_i)_{i=1}^n\in \Wbf(D)$ be a cycle in $D$.
  The \deftext[lifting of c in D@lifting of $c$ in $D$]{lifting of $\bm c$ in $\bm D$ by $\bm(\bm x\bm,\bm t\bm)$} is the digraph
  $D^{(c)}_{(x,t)} = (V\disunion\SET{x,t}, A')$ where
  \( A' = A \BSET{(c_1,c_2)} \cup \SET{(c_1,t),(x,c_2),(x,t)}. \qedhere\)
\end{definition}

\noindent 
\begin{tabular}{p{6.5cm}p{6.5cm}}
~~Observe that the cycle $c\in\Wbf(D)$ is no longer a walk  with respect to  the lifting of $c$ in $D$ anymore.
&
%  Consider $D=(\SET{c_1,c_2,c_3,c_4},\SET{(c_1,c_2),(c_2,c_3),(c_3,c_4),(c_4,c_1)})$. Then
%  $c_1c_2c_3c_4c_1\in\Wbf(D)$ is a cycle. The lifting of $c$ in $D$ by $(x,t)$ is then defined to be
%  the digraph $D'=(\SET{c_1,c_2,c_3,c_4,x,t},\SET{(c_1,t),(c_2,c_3),(c_3,c_4),(c_4,c_1),(x,c_2),(x,t)})$. \\[5mm]
%  \hspace*{4cm}
  \includegraphics[valign=t,scale=.85]{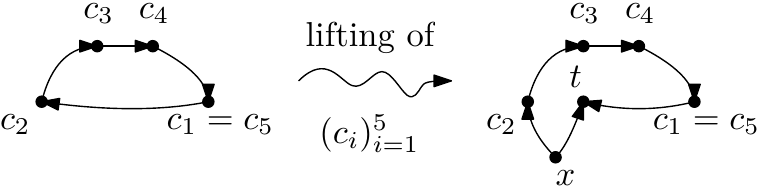} 
\end{tabular}

Clearly, if $c'=(c'_i)_{i=1}^n\in \Wbf(D')$ is a cycle walk in the lifting $D'$ of the cycle $c$ of $D$,
then $c'\in \Wbf(D')$, too. Thus lifting of cycles strictly decreases the number of cycles in the digraph.

\begin{definition}\label{def:completeLifting}\PRFR{Feb 15th}
  Let $D=(V,A)$ be a digraph. A \deftext[complete lifting of D@complete lifting of $D$]{complete lifting of $\bm D$}
  is an acyclic digraph $D'=(V',A')$ for which there is a suitable $n\in \N$ such that there is a
  set $X=\dSET{x_1,t_1,x_2,t_2,\ldots,x_n,t_n}$ with $X\cap V = \emptyset$,
  a family of digraphs $D^{(i)} = (V^{(i)},A^{(i)})$ for $i\in \SET{0,1,\ldots,n}$
  where $D' = D^{(n)}$, $D^{(0)} = D$, and for all $i\in\SET{1,2,\ldots,n}$
   $$D^{(i)} = \left( D^{(i-1)}\right)^{(c_i)}_{(x_i,t_i)}$$
  with respect to a cycle walk $c_i\in \Wbf\left( D^{(i-1)}\right)$.
%  In this case, we say that  $R = \SET{(x_i,t_i)\mid i\in\SET{1,2,\ldots,n}}$ \deftextX{realizes} 
%  the complete lifting $D'$ of $D$.
\end{definition}

\begin{lemma}\label{lem:completelifting}\PRFR{Feb 15th}
  Let $D=(V,A)$ be a digraph. Then $D$ has a complete lifting.
\end{lemma}
\begin{proof}\PRFR{Feb 15th}
  By induction on the number of cycle walks in $D$, lifting an arbitrarily chosen cycle walk yields
  a digraph with strictly less cycles, and every acyclic digraph is its own complete lifting.
\end{proof}

\needspace{3\baselineskip}
\begin{lemma}\label{lem:cyclelifting}\PRFR{Feb 15th}
  Let $D=(V,A)$, $E,T\subseteq V$, $c\in \Wbf(D)$ a cycle, $x,t\notin V$, and let $D'=D^{(c)}_{(x,t)}$ be the lifting of $c$ in $D$.
  Then $\Gamma(D,T,E) = \Gamma(D',T\cup\SET{t},E\cup\SET{x})\contract E$.
\end{lemma}

\begin{figure}[b]
\caption{\label{fig:liftingcycles}Constructions involved in Lemma~\ref{lem:cyclelifting}.}
\begin{center}
\includegraphics{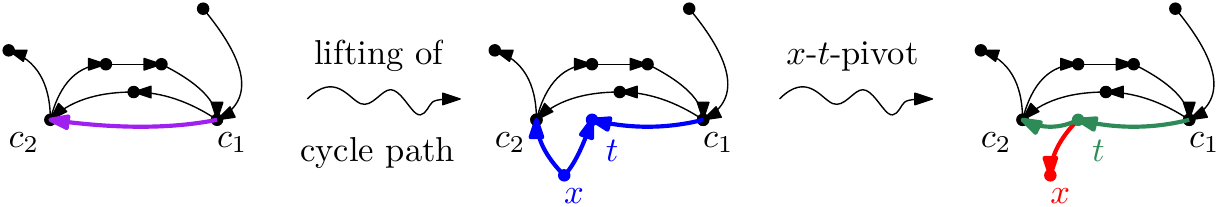}
\end{center}
\end{figure}

\begin{proof}\PRFR{Feb 15th}
  Let $M=\Gamma(D,T,V)$ be the strict gammoid induced by the representation $(D,T,E)$ 
  of the %not necessarily strict 
  gammoid $\Gamma(D,T,E)$, and let 
  $M' = \Gamma(D',T\cup\SET{t},V')$ be the strict gammoid obtained from the lifting of $c$.
  Then $M'' = M' \contract V\cup\SET{t} $ is a strict gammoid that is represented by $(D'',T,V\cup\SET{t})$
  where the digraph
   $D'' = (V_0\BSET{x},A'')$ with $A'' = A_0\BS\left( V_0\times\SET{x} \right)$
  and where % the $x$-$t$-pivot $D_0$ of $D'$, i.e.
   $D'_{x\leftarrow t} = (V_0,A_0)$.
   % This follows from the proof of Lemma~\ref{lem:contractionStrictGammoid}
  % along with the single-arc routing $\SET{xt}\colon \SET{x}\routesto T\cup\SET{t}$ in $D'$.
  %
  It is easy to see from the involved constructions (Fig.~\ref{fig:liftingcycles}), 
  that $A'' = \left(  A\BSET{(c_1,c_2)} \right) \cup\SET{(c_1, t), (t, c_2)}$.
   A routing $R$ in $D$ can have at most one path $p\in R$ such that $(c_1,c_2)\in \left| p \right|_A$, 
   and since $t\notin V$, we obtain a routing $R'= \left(R\BSET{p}\right)\cup\SET{q t r}$ 
   for $q,r\in \Pbf(D)$ such that 
  $p=qr$  with $q_{-1}=c_1$ and $r_1=c_2$.
   Clearly, $R'$ routes $X$ to $Y$ in $D''$ whenever $R$ routes $X$ to $Y$ in $D$.
  Conversely, let $R'\colon X'\routesto Y'$ be a routing in $D''$ with $t\notin X'$.
   Then there is at most one $p\in R'$ with $t\in \left| p \right|$.
    We can invert the construction and let $R''= \left( R'\BSET{p}  \right)\cup\SET{qr}$
     for the appropriate paths $q,r\in \Pbf(D'')$ with $p=qtr$. 
  Then $R''$ is a routing from $X'$ to $Y'$ in $D'$. Thus $M''\restrict V = M$. Consequently,
  \begin{align*}
  \Gamma(D,T,E) = M\restrict E = \left( M''  \right)\restrict E & = \left( \Gamma(D',T\cup\SET{t},V') \contract \left( V\cup\SET{t} \right) \right) \restrict E \\& = \Gamma(D',T\cup\SET{t},E\cup\SET{x})\contract E. \qedhere
\end{align*}
\end{proof}

\begin{corollary}\label{cor:acyclicQuasiRepresentationOfGammoids}\PRFR{Feb 15th}
  Let $M=(E,\Ical)$ be a gammoid. Then there is an acyclic digraph $D=(V,A)$ and sets $T,E'\subseteq V$ such that
  $M = \Gamma \left( D,T,E' \right)\contract E$
  and $\left| T \right| = \rk_M(E) + \left| E'\BS E \right|$.
\end{corollary}
\begin{proof}\PRFR{Feb 15th}
  Let $M=\Gamma(D',T',E)$ with $\left| T' \right|=\rk_M(E)$ --- such a representation may be obtained from any other representation by
  adding the appropriate amount of new targets to the digraph, and connecting every new target with every element from the old target set.
  Then let $D$ be a complete lifting of $D'$ (Lemma~\ref{lem:completelifting}),
  and let $D^{(0)},D^{(1)},\ldots, D^{(n)}$ be the family of digraphs and $c_1,c_2,\ldots,c_n$ be the cycle walks that correspond to 
  the complete lifting $D$ of $D'$ 
  as required by Definition~\ref{def:completeLifting},
  and let $\dSET{x_1,t_1,\ldots,x_n,t_n}$ denote the new elements such that
  $ D^{(i)} = \left( D^{(i-1)}  \right)^{(c_i)}_{(x_i,t_i)} $
  holds for all $i\in\SET{1,2,\ldots,n}$.
  Induction on the index $i$ with Lemma~\ref{lem:cyclelifting} yields that
  $ \Gamma(D',T,E) = \Gamma(D^{(i)},T\cup\SET{t_1,t_2,\ldots,t_i},E\cup\SET{x_1,x_2,\ldots,x_i})\contract E$
  holds for all $i\in\SET{1,2,\ldots,n}$.
  Clearly,
  $ \left| T\cup\SET{t_1,t_2,\ldots,t_n} \right| = \left| T \right| + n = \rk_M(E) + n = \rk_M(E) + \left| \SET{x_1,x_2,\ldots,x_n} \right| $.
\end{proof}

Since the contraction $\Ocal \contract X$ (\cite{BlVSWZ99}, Prop.~3.3.2) of an oriented matroid $\Ocal$ is an orientation of the contraction $M(\Ocal)\contract X$ of its
underlying matroid, we are able to obtain heavy arc orientations of 
gammoids that cannot be represented without cycles in their digraphs through
complete lifting.

Since every heavy arc orientation of a gammoid is representable, an open question that occurs naturally is, whether there is a similar
combinatorial way that also yields non-representable orientations of gammoids. Furthermore, is there a way to refine the definition of
 heavy arc orientations that allows to circumvent the formation of a complete lifting?

% ACKNOWLEDGEMENTS

\bigskip
\footnotesize
\noindent\textit{Acknowledgments.}
This research was partly supported by a scholarship granted by the FernUniversität in Hagen.

% BIBLIOGRAPHY

\bibliographystyle{plain}

\bibliography{references.bib}

\end{document}